\documentclass[12pt]{article}

\usepackage{amssymb}
\usepackage[cal=boondox]{mathalfa}
\usepackage{amsthm}
\usepackage{amsmath}
\usepackage{graphicx}
\usepackage{fullpage}
\usepackage{color}
\usepackage{enumerate}
 \numberwithin{equation}{section}
\usepackage{booktabs}
\allowdisplaybreaks

 \usepackage{mathtools}
 \mathtoolsset{showonlyrefs}
\usepackage[nocompress]{cite}

\theoremstyle{plain}
\newtheorem{thm}{Theorem}[section]
\newtheorem{cor}[thm]{Corollary}

\newtheorem{lem}[thm]{Lemma}

\theoremstyle{definition}

\theoremstyle{remark}

\newtheorem{rem}[thm]{Remark}

\newcommand{\N}{\mathbb{N}}
\newcommand{\R}{\mathbb{R}}

\newcommand{\bp}{\begin{proof}[\ensuremath{\mathbf{Proof}}]}
\newcommand{\bs}{\begin{proof}[\ensuremath{\mathbf{Solution}}]}
\newcommand{\ep}{\end{proof}}
\newcommand{\be}{\begin{equation}}
\newcommand{\ee}{\end{equation}}

\begin{document}

\title{Two critical times for the SIR model}

\author{Ryan Hynd\footnote{Department of Mathematics, University of Pennsylvania. }\;, Dennis Ikpe\footnote{Department of Statistics and Probability, Michigan State University. African Institute for Mathematical Sciences, South Africa.}\;, and Terrance Pendleton\footnote{Department of Mathematics, Drake University.}\;}

\maketitle 

\begin{abstract} 
We consider the SIR model and study the first time the number of infected individuals begins to decrease and the first time this population is below a given threshold. We interpret these times as functions of the initial susceptible and infected populations and characterize them as solutions of a certain partial differential equation. This allows us to obtain integral representations of these times and in turn to estimate them precisely for large populations.  
\end{abstract}

\section{Introduction}
The susceptible, infected, and recovered (SIR) model in epidemiology involves the system of ODE 
\be\label{SIR}
\begin{cases}
\dot S=-\beta S I\\
\dot I=\beta SI -\gamma I.
\end{cases}
\ee 
Here $S,I: [0,\infty)\rightarrow [0,\infty)$ denote the susceptible and infected compartments of a given population in the presence of an infectious disease. If $N$ is the size of the population, then 
$$
R(t)=N-S(t)-I(t)
$$
is the recovered compartment of the population at time $t$. The parameters $\beta>0$ and $\gamma>0$ are the  infected and recovery rates per unit time, respectively. 

\par We note that for given initial conditions $S(0), I(0)\ge 0$, there is a unique solution of the SIR ODE. Indeed, a local solution pair $S,I: [0,T)\rightarrow \R$ of \eqref{SIR} exists by standard ODE theory (see for example section I.2 \cite{MR587488}). As 
\be\label{ExpFormulae}
S(t)= S(0)e^{\displaystyle-\beta \int^t_0I(\tau)d\tau}\;\text{and}\;\; I(t)=I(0)e^{\displaystyle \int^t_0(\beta S(\tau)-\gamma)d\tau}
\ee
for $t\in [0,T)$, $S$ and $I$ are nonnegative. By \eqref{SIR}, $\frac{d}{dt}(S+I)(t)=-\gamma I(t)\le 0$. Therefore,
$$
S(t)+I(t)\le S(0)+I(0)
$$
for all $t\in[0,T)$.  As a result, $S$ and $I$ are uniformly bounded. Consequently it is possible to continue this solution to all of $[0,\infty)$ (section I.2 \cite{MR587488}).  And in view of \eqref{ExpFormulae}, $S(t)$ is positive for all $t>0$ provided $S(0)>0$ and likewise for $I(t)$.

\par Let us recall three important properties about each solution $S,I: [0,\infty)\rightarrow (0,\infty)$ of the SIR ODE.  For more details on these facts, 
we refer the reader to section 2.2 of \cite{MR3409181} and section 9.2 of \cite{MR3969982}. 
\\\\
\noindent {\bf Integrability}. Although $S$ and $I$ are not explicitly known, they are integrable in sense that 
\be\label{LevelSetFun}
S(t)+I(t)-\frac{\gamma}{\beta}\ln S(t)=S(0)+I(0)-\frac{\gamma}{\beta}\ln S(0)
\ee
for each time $t\ge 0$. That is, the path $t\mapsto (S(t),I(t))$ belongs to a level set of the function 
$$
\psi(x,y):=x+y-(\gamma/\beta)\ln x.
$$
In particular, we may consider $I$ as a function of $S$.  
\\\\
\noindent {\bf Decay of  infected individuals}.
The number of infected individuals tends to $0$ as $t\rightarrow\infty$ 
\be\label{ItendtoZero}
\lim_{t\rightarrow\infty}I(t)=0.
\ee
In particular, for any given threshold $\mu>0$, there is a finite time $t$ such that the number of infected individuals falls below $\mu$
$$
I(t)\le \mu. 
$$
In what follows, we will write $u\ge 0$ for the first time in which $I$ falls below $\mu$. 
\\\\
\noindent {\bf Infected individuals eventually decrease}. There is a time $v\ge 0$ for which 
$$
\text{$I$ is decreasing on $[v,\infty)$.}
$$Since $I$ is a positive function and $\dot I(t)=(\beta S(t)-\gamma)I(t)$, $v$ can be taken to be the first time $t$ that the number of susceptible individuals falls below the ratio of the recovery and infected rates
\be\label{Slessthangammaoverbeta}
 S(t)\le \frac{\gamma}{\beta}.
\ee
\begin{figure}[h]
\centering
\includegraphics[width=.8\textwidth]{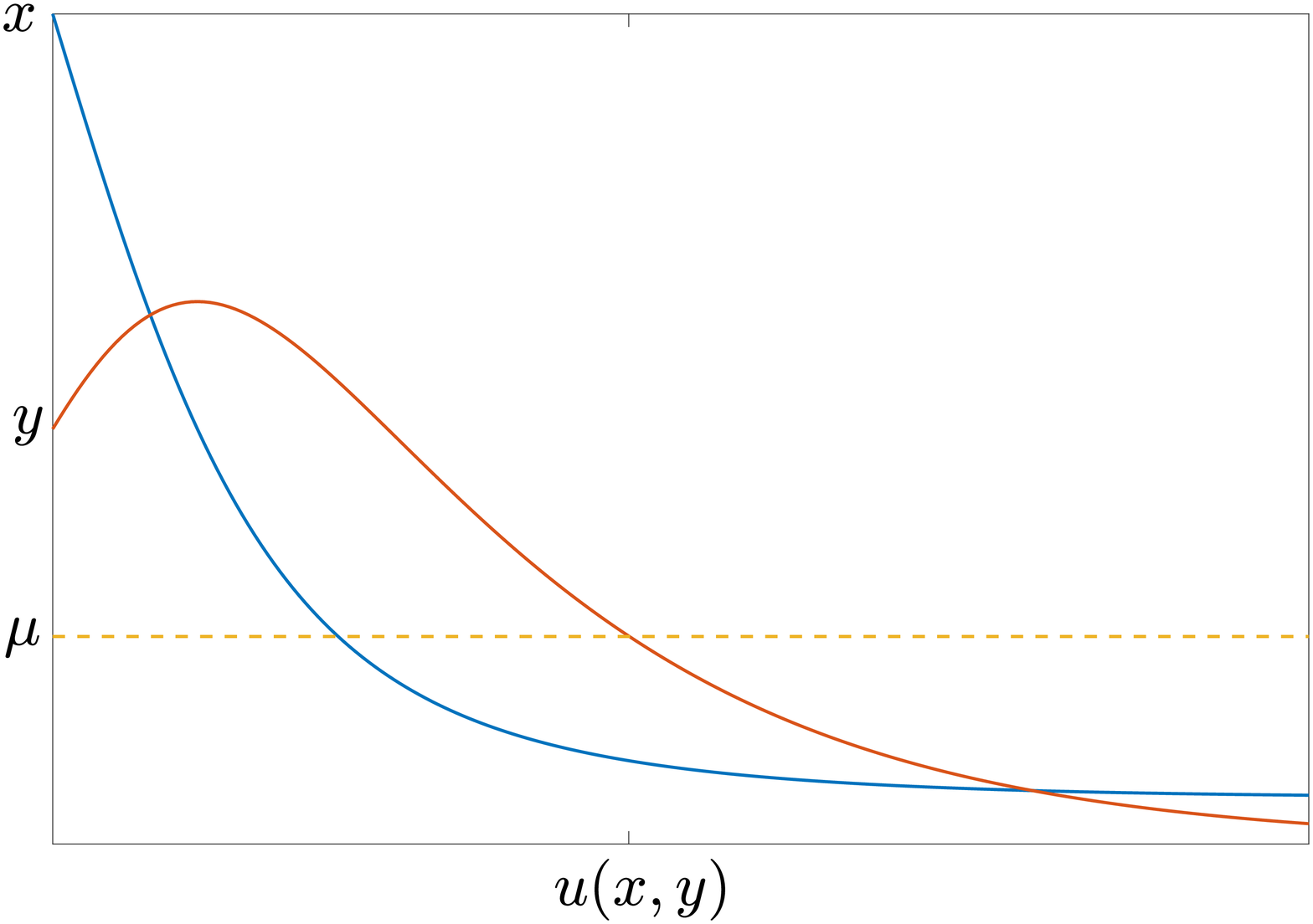}
 \caption{Plot of the solution pair $S,I$ of \eqref{SIR} with $S(0)=x$ and $I(0)=y$. $S(t)$ is shown in blue and $I(t)$ is shown in red. Note that $u(x,y)$ is the first time $t$ such that $I(t)\le \mu$. }
\end{figure}
\par In this note, we will study the times $u$ and $v$ mentioned above as functions of the initial conditions $S(0)$ and $I(0)$. First, we will fix a threshold
$$
\mu>0
$$
and consider
\be\label{uExitTimeFun}
u(x,y)
:=\min\{t\ge 0: I(t)\le \mu\}
\ee
for $x,y\ge 0$. Here $S, I$ are solutions of \eqref{SIR} with $S(0)=x$ and $I(0)=y$. As $u(x,y)=0$ when $0\le y< \mu$, we will focus on the values of $u(x,y)$ for $x\ge 0$ and $y\ge \mu$.

\par We'll see that $u$ is a smooth function on $(0,\infty)\times(\mu,\infty)$ which satisfies the PDE
\be\label{babyHJB}
\beta xy \partial_xu+(\gamma-\beta x)y\partial_yu=1
\ee
and boundary condition
\be\label{uBC}
u(x,\mu)=0,\; x\in [0,\gamma/\beta].
\ee
Moreover, we will also use $\psi$ to write a representation formula for $u$. To this end, we note that for $x> 0,y\ge \mu$ and there is a unique $a(x,y)\in (0,\gamma/\beta]$ such that 
$$
\psi(x,y)=\psi(a(x,y),\mu).
$$ 
Further, $a(x,y)<x$ when $x>\gamma/\beta$ or $y>\mu$. These fact follows easily from the definition of $\psi$; Figure \ref{LevelSetPlot} below also provides a schematic.

\par A basic result involving $u$ is as follows. 

\begin{thm}\label{uthm}
The function $u$ defined in \eqref{uExitTimeFun} has the following properties. 

\begin{enumerate}[(i)]

\item $u$ is continuous on $[0,\infty)\times[\mu,\infty)$ and is smooth in $(0,\infty)\times(\mu,\infty)$. 

\item $u$ is the unique solution of \eqref{babyHJB} in $(0,\infty)\times(\mu,\infty)$ which satisfies the boundary condition \eqref{uBC}.

\item  For each $x>0$ and $y\ge \mu$,
 \be\label{RepFormulaForU}
u(x,y)=\int^x_{a(x,y)}\frac{dz}{\beta z\left((\gamma/\beta)\ln z-z +\psi(x,y)\right)}.
\ee
\end{enumerate}
 \end{thm}

\par Next, we will study 
\be\label{vExitTimeFun}
v(x,y):=\min\{t\ge 0: S(t)\le \gamma/\beta\}
\ee
for $x\ge 0$ and $y\ge0$.  Above, we are assuming that $S, I$ is the solution pair of the SIR ODE \eqref{SIR} with $S(0)=x$ and $I(0)=y$. Note that $v(x,y)$ records the first time $t$ that $I(t)$ starts to decrease. Since 
\be
\begin{cases}
v(x,y)=0\;\text{ for $0\le x\le \gamma/\beta$ and}\\\\
v(x,0)=\infty\;\text{ for $x> \gamma/\beta$},
\end{cases}
\ee
we will focus on the values of $v(x,y)$ for $x>\gamma/\beta$ and $y> 0$.  
\begin{figure}[h]
\centering
 \includegraphics[width=.8\textwidth]{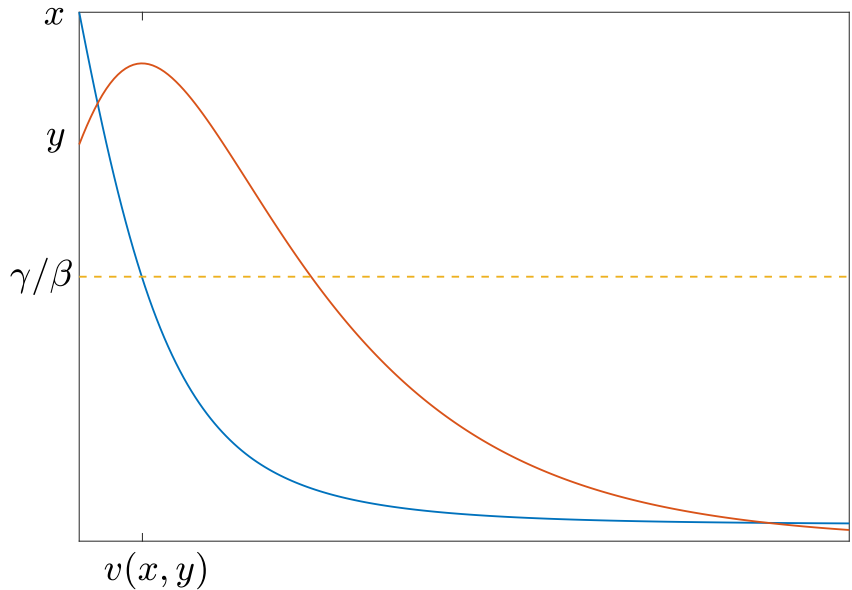}
 \caption{The solution pair $S,I$ of \eqref{SIR} with $S(0)=x$ and $I(0)=y$. $S(t)$ is shown in blue and $I(t)$ is shown in red. Note that $v(x,y)$ is the first time $S(t)\le \gamma/\beta$ which is also the same time that $I$ starts to decrease.}\label{Fig6}
\end{figure}
\par The methods we use to prove Theorem \ref{uthm} extend analogously to $v$. In particular, $v$ satisfies the same PDE as $u$ with the 
boundary condition
\be\label{vBC}
v(\gamma/\beta,y)=0,\; y\in (0,\infty).
\ee
Almost in parallel with Theorem \ref{uthm}, we have the subsequent assertion. 

\begin{thm}\label{vthm}
The function $v$ defined in \eqref{vExitTimeFun} has the following properties. 

\begin{enumerate}[(i)]

\item $v$ is continuous on $[\gamma/\beta,\infty)\times(0,\infty)$ and is smooth in $(\gamma/\beta,\infty)\times(0,\infty)$. 

\item $v$ is the unique solution of \eqref{babyHJB} in $(\gamma/\beta,\infty)\times(0,\infty)$ which satisfies the boundary condition \eqref{vBC}.

\item  For each $x\ge \gamma/\beta$ and $y>0$,
 \be\label{RepFormulaForV}
v(x,y)=\int^x_{\gamma/\beta}\frac{dz}{\beta z\left((\gamma/\beta)\ln z-z +\psi(x,y)\right)}.
\ee

\end{enumerate}
 \end{thm}
 
 \par We will then use the above representation formulae to show how to precisely estimate the $u(x,y)$ and $v(x,y)$ when $x+y$ is large.  

\begin{thm}\label{AsympThm}
Define $u$ by \eqref{uExitTimeFun} and $v$ by \eqref{vExitTimeFun}. Then 
\be\label{uAsymptotic}
\lim_{\substack{x+y\rightarrow\infty \\ x\ge 0,\;y\ge \mu}}\frac{\displaystyle u(x,y)}{\displaystyle\frac{1}{\gamma}\ln\left(\frac{x+y}{\mu}\right)}=1
\ee
and 
\be\label{vAsymptotic} 
\lim_{\substack{x+y\rightarrow\infty \\ x> \gamma/\beta,\;y>0}}\frac{\displaystyle \beta(x-\gamma/\beta+y) }{\displaystyle\ln\left[\left(\frac{x}{\gamma/\beta}\right)\left(\frac{x-\gamma/\beta}{y}+1\right)\right]}\cdot v(x,y)=1. 
\ee
\end{thm}
 
 \par The limit \eqref{uAsymptotic} implies that $u(x,y)$ tends to $\infty$ as $x+y\rightarrow\infty$. This limit also reduces to an exact formula when $x=0$. In this case, $S(t)=0$ and $I(t)=ye^{-\gamma t}$, so 
 $$
 u(0,y)=\frac{1}{\gamma}\ln\left(\frac{y}{\mu}\right).
 $$
 Alternatively, the limit \eqref{vAsymptotic} implies that $v(x,y)$ tends to $0$ as $x+y\rightarrow\infty$ with $y\ge \delta$ for each $\delta>0$.  This will be a crucial element of our proof of \eqref{uAsymptotic}.
 
 \par This paper is organized as follows. In section \ref{uSect}, we will study $u$ and prove Theorem \ref{uthm}.  Then in section \ref{vSect}, we will indicate what changes are necessary so that our proof of Theorem \ref{uthm} adapts to Theorem \ref{vthm}.  Finally, we will prove Theorem \ref{AsympThm} in section \ref{AsympSect}. We also would like to acknowledge that this material is based upon work supported by the NSF
under Grants No. DMS-1440140 and DMS-1554130, NSA under Grant No.
H98230-20-1-0015, and the Sloan Foundation under Grant No. G-2020-12602 while the
authors participated in the ADJOINT program hosted by MSRI.

\section{The first time $I(t)\le \mu$}\label{uSect}
This section is dedicated to proving Theorem \ref{uthm}.  We will begin with an elementary upper bound on $u$. 

\begin{lem}\label{CrudeUpperU}
For each $x\ge 0$ and $y\ge \mu$, 
\be
u(x,y)\le \frac{x+y}{\gamma\mu}.
\ee
\end{lem}
\begin{proof}
Let $S,I$ be the solution pair of the SIR ODE \eqref{SIR} with $S(0)=x$ and $I(0)=y$. Since $\frac{d}{dt}(S+I)(t)=-\gamma I(t)\le 0$,  
$$
\int^t_0 \gamma I(\tau)d\tau +S(t)+I(t)=x+y.
$$ 
Choosing $t=u(x,y)$ and noting that $I(\tau)\ge \mu$ for $0\le \tau \le u(x,y)$ gives 
$$
u(x,y)\gamma\mu\le \int^{u(x,y)}_0 \gamma I(\tau)d\tau\le x+y.
$$ 
\end{proof}
Next we observe that $I$ is always decreasing at the time it reaches the threshold $\mu$. 
 \begin{lem}\label{NondegDerivativelem}
Let $x>0$, $y>\mu$, and suppose $S,I$ is the solution of \eqref{SIR} with $S(0)=x$ and $I(0)=y$.  Then 
$$
\dot I(u(x,y))<0. 
$$
\end{lem}
\begin{proof}
As $\dot I(t)=(\beta S(t)-\gamma)I(t)$, it suffices to show 
\be\label{SIneqIFT}
\beta S(u(x,y))-\gamma<0.
\ee
If $\beta x\le \gamma$, then $\beta S(t)<\gamma$ for all $t>0$ since $S$ is decreasing. Consequently, \eqref{SIneqIFT} holds for $t=u(x,y)$.  Alternatively, if $\beta x> \gamma$, then $I$ is nonincreasing on the interval $[0, v(x,y)]$ and
$$
\beta S(v(x,y))-\gamma=0.
$$ 
In particular, $I(v(x,y))\ge y>\mu$. Thus, $v(x,y)<u(x,y)$.  As $S$ is decreasing, 
$$
\beta S(u(x,y))-\gamma<\beta S(v(x,y))-\gamma=0.
$$
\end{proof}
\begin{rem}
Since $v(x,y)=0$ for $x\in [0,\gamma/\beta]$, it also follows from this proof that 
\be
v(x,y)\le u(x,y)
\ee
for each $x\ge0$ and $y\ge \mu$.
\end{rem}
We will also need to verify that solutions of the SIR ODE \eqref{SIR} depend continuously on their initial conditions. 
\begin{lem}\label{UniformLimSkIkLem}
Suppose $x^k\ge 0$ and $y^k\ge 0$, and let $S^k,I^k$ be the solution of \eqref{SIR} with $S^k(0)=x^k$ and $I^k(0)=y^k$ for each $k\in \N$. 
If $x^k\rightarrow x$ and $y^k\rightarrow y$ as $k\rightarrow\infty$, then 
\be\label{UniformLimSkIk}
S(t)=\lim_{k\rightarrow\infty}S^k(t_k)\quad\text{and}\quad I(t)=\lim_{k\rightarrow\infty}I^k(t_k)
\ee
for each sequence $t_k\ge 0$ such that $t_k\rightarrow t$. Here $S,I$ is the solution of \eqref{SIR} with $S(0)=x$ and $I(0)=y$.
\end{lem}
\begin{proof}
We note that $S^k$ and $I^k$ are nonnegative functions with
$$
S^k(t)+I^k(t)\le x^k+y^k
$$
for all $t\ge 0$. Thus, $S^k$ and $I^k$ are uniformly bounded independently of $k\in \N$. In view of \eqref{SIR}, we additionally have  
$$
|\dot S^k(t)|\le \beta(x^k+y^k)^2
$$
and 
$$
|\dot I^k(t)|\le \beta(x^k+y^k)^2+\gamma(x^k+y^k)
$$
for each $t\ge 0$. Consequently, $S^k, I^k$ are uniformly equicontinuous on $[0,\infty)$. 

\par The Arzel\`a-Ascoli theorem then implies there are subsequences $S^{k_j}$ and $I^{k_j}$ converging uniformly on any bounded subinterval of $[0,\infty)$ to continuous functions $S$ and $I$, respectively.  
Observe 
\be
S^k(t)=x^k-\int^t_0\beta S^k(\tau)I^k(\tau)d\tau\;\;\text{and}\;\; I^k(t)=y^k+\int^t_0(\beta S^k(\tau)-\gamma)I^k(\tau)d\tau
\ee
for each $t\ge 0$. Letting $k=k_j\rightarrow\infty$ and employing the uniform convergence of $S^{k_j}$ and $I^{k_j}$ on $[0,t]$ for each $t\ge 0$, we see that $S,I$ is the solution of \eqref{SIR} with $S(0)=x$ and $I(0)=y$. As this limit is independent of the 
subsequence, it must be that $S^k$ and $I^k$ converge to $S$ and $I$, respectively, locally uniformly on $[0,\infty)$. As a result, we conclude \eqref{UniformLimSkIk}.
\end{proof}
In our proof of Theorem \ref{uthm} below, we will employ the flow of the SIR ODE \eqref{SIR}. This is the mapping 
\be
\Phi:[0,\infty)^3\rightarrow [0,\infty)^2; (x,y,t)\mapsto (S(t),I(t))
\ee
where $S,I$ is the solution pair of \eqref{SIR} with $S(0)=x$ and $I(0)=y$.  We will also write $\Phi=(\Phi_1,\Phi_2)$ so that
\be
\Phi_1(x,y,t)=S(t)\;\;\text{and}\;\;\Phi_2(x,y,t)=I(t).
\ee
A direct corollary of Lemma \ref{UniformLimSkIkLem} is that $\Phi$ is a continuous mapping. With a bit more work, it can also be shown that $\Phi: (0,\infty)^3\rightarrow [0,\infty)^2$ is smooth (exercise 3.2 in Chapter 1 of \cite{MR587488}, Chapter 1 section 7 of \cite{MR0069338}).

\begin{proof}[Proof of Theorem \ref{uthm}]
$(i)$ Suppose $x^k\ge 0$ and $y^k\ge \mu$ with $x^k\rightarrow x$ and $y^k\rightarrow y$ as $k\rightarrow\infty$.  By Lemma \eqref{CrudeUpperU}, $u(x^k,y^k)$ is uniformly bounded.  We can then select a subsequence $u(x^{k_j},y^{k_j})$ which converges to some $t\ge0$. From the definition of $u$, we also have 
\be\label{Phi2xkykukEqn}
\Phi_2(x^k,y^k,u(x^k,y^k))=\mu
\ee
for each $k\in \N$. Since $\Phi_2$ is continuous, we can send $k=k_j\rightarrow\infty$ in \eqref{Phi2xkykukEqn} to get 
$$
\Phi_2(x,y,t)=\mu.
$$
As $y\ge \mu$, the limit $t$ is equal to $u(x,y)$. Because this limit is independent of the subsequence $u(x^{k_j},y^{k_j})$, it must be that 
\be
u(x,y)=\lim_{k\rightarrow\infty}u(x^k,y^k).
\ee
It follows that $u$ is continuous on $[0,\infty)\times[\mu,\infty)$.

\par Let $x>0$ and $y>\mu$ and recall that $\Phi_2(x,y,u(x,y))=\mu$.  By Lemma \ref{NondegDerivativelem}, 
\be
\partial_t\Phi_2(x,y,u(x,y))<0.
\ee
Since $\Phi_2$ is smooth in a neighborhood of $(x,y,u(x,y))$, the implicit function theorem implies that $u$ is smooth in a 
neighborhood of $(x,y)$.  We conclude that $u$ is smooth in $(0,\infty)\times (\mu,\infty)$.
\\
\par $(ii)$ Fix $x>0$ and $y>\mu$, and let $S,I$ be the solution of \eqref{SIR} with $S(0)=x$ and $I(0)=y$.
Observe that for each $0\le t< u(x,y)$, 
\begin{align*}
u(S(t),I(t))&=\min\{\tau\ge 0: I(t+\tau)\le \mu\}\\
&=\min\{s\ge t: I(s)\le \mu\}-t\\
&=u(x,y)-t.
\end{align*}
Therefore,
\begin{align}
1&=
-\left.\frac{d}{dt}u(S(t),I(t))\right|_{t=0}\\
&=\left.\Big[\beta S(t)I(t)\partial_xu(S(t),I(t))+(\gamma-\beta S(t))I(t)\partial_yu(S(t),I(t))\Big]\right|_{t=0}\\
&=\beta xy \partial_xu(x,y)+(\gamma-\beta x)y\partial_yu(x,y).
\end{align}
We conclude that $u$ satisfies \eqref{babyHJB}. \par Now suppose $w$ is a solution of \eqref{babyHJB} which satisfies the boundary condition \eqref{uBC}. Note  
$$
\frac{d}{dt}w(S(t),I(t))=-\Big[\beta S(t)I(t)\partial_xw(S(t),I(t))+(\gamma-\beta S(t))I(t)\partial_yw(S(t),I(t))\Big]=-1
$$
for $0\le t<u(x,y)$. Integrating this equation from $t=0$ to $t=u(x,y)$ gives 
$$
w(S(u(x,y)),I(u(x,y)))-w(x,x)=-u(x,y).
$$
Since $I(u(x,y))=\mu$, $\beta S(u(x,y))\le \gamma$, and $w(S(u(x,y)),\mu)=0$, it follows that $w(x,y)=u(x,y)$. Therefore, $u$ is the unique solution of 
the PDE \eqref{babyHJB} which satisfies the boundary condition \eqref{uBC}.
\\
\par $(iii)$ Suppose either $x>0$ and $y>\mu$ or $x>\gamma/\beta$ and $y=\mu$.  We recall that $a=a(x,y)\in (0,\gamma/\beta]$ is the unique solution of 
$$
\psi(x,y)=\psi(a,\mu).
$$
It is also easy to check that 
$$
\frac{\gamma}{\beta}\ln z-z+\psi(x,y)>\mu
$$
for $a<z<x$.  See Figure \ref{LevelSetPlot} for an example.
\begin{figure}[h]
\centering
 \includegraphics[width=.8\textwidth]{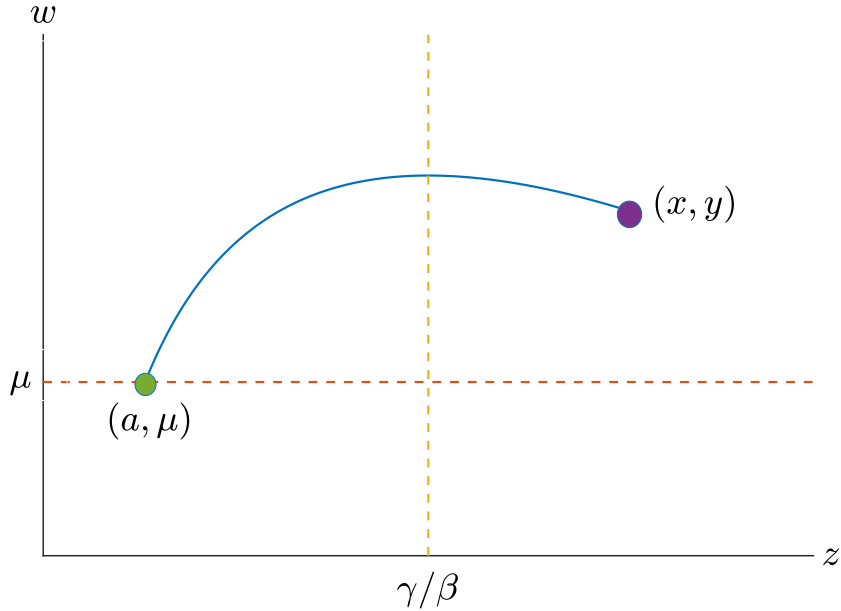}
 \caption{Graph of $w=(\gamma/\beta)\ln z-z+\psi(x,y)$ for $a\le z\le x$. Here $x>0$, $y>\mu$, and $a=a(x,y)\in (0,\gamma/\beta]$ is the unique solution of $\psi(x,y)=\psi(a,\mu)$. Also note that 
 this graph is a subset of the level set $\{(z,w)\in (0,\infty)^2: \psi(z,w)=\psi(x,y)\}$.}\label{LevelSetPlot}
\end{figure}

\par Since $u$ solves the PDE \eqref{babyHJB}, 
\begin{align}
\frac{d}{dz}u\left(z,\frac{\gamma}{\beta}\ln z-z+\psi(x,y)\right)&=\left.\frac{\beta zy \partial_xu(z,y)+(\gamma-\beta z)y\partial_yu(z,y)}{\beta zy}\right|_{y=\frac{\gamma}{\beta}\ln z-z+\psi(x,y)}\\
&=\frac{1}{\beta z\left(\frac{\gamma}{\beta}\ln z-z+\psi(x,y)\right)}
\end{align}
for $a<z<x$. Integrating from $z=a$ to $z=x$ and using the boundary condition \eqref{uBC} gives
$$
u(x,y)=\int^x_{a}\frac{dz}{\beta z\left((\gamma/\beta)\ln z-z +\psi(x,y)\right)}.
$$
\end{proof}
\begin{figure}[h]
\centering
 \includegraphics[width=.8\textwidth]{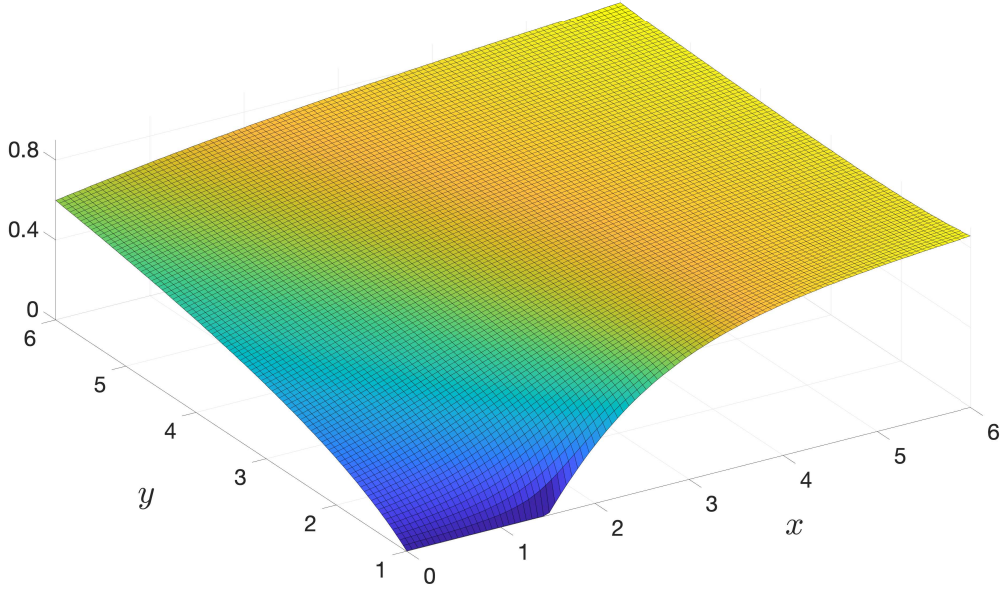}
 \caption{Numerical approximation of the function $u(x,y)$ for $(x,y)\in [0,6]\times[1,5]$. Here $\mu=1$, $\beta=2$, and $\gamma=3$.}\label{Fig4}
\end{figure}

\section{The first time $S(t)\le \gamma/\beta$}\label{vSect}
We will briefly point out what needs to be adapted from the previous section so that we can conclude Theorem \ref{vthm} involving $v$. 
We first note that $v$ is locally bounded in $[\gamma/\beta,\infty)\times (0,\infty)$.

\begin{lem}\label{CrudeUpperV}
For each $x\ge \gamma/\beta$ and $y>0$, 
\be
v(x,y)\le \frac{\ln x-\ln(\gamma/\beta)}{\beta y}.
\ee
\end{lem}
\begin{proof}
Let $S,I$ denote the solution of \eqref{SIR} with $S(0)=x$ and $I(0)=y$. Since $I(t)$ is increasing on $t\in [0,v(x,y)]$, $I(t)\ge y$ for $t\in [0,v(x,y)]$. 
In view of \eqref{ExpFormulae}
\be
\frac{\gamma}{\beta}=S(v(x,y))=xe^{\displaystyle-\beta \int^{v(x,y)}_0I(\tau)d\tau}\le xe^{-\beta y v(x,y)}.
\ee
Taking the natural logarithm and rearranging leads to $v(x,y)\le  (\ln x -\ln(\gamma/\beta))/\beta y $, which is what we wanted to show.
\end{proof}
The next assertion follows since $S$ is decreasing whenever $I$ is initially positive. The main point of stating this lemma is to make an analogy with Lemma \ref{NondegDerivativelem}. 
 \begin{lem}\label{NondegDerivativeSlem} 
Let $x\ge \gamma/\beta$ and $y>0$, and suppose $S,I$ is the solution of \eqref{SIR} with $S(0)=x$ and $I(0)=y$.  Then 
\be
\dot S(v(x,y))<0. 
\ee
\end{lem}
Having established Lemmas \ref{CrudeUpperV} and \ref{NondegDerivativeSlem}, we can now argue virtually the same way we did in the previous section to conclude Theorem \ref{vthm}. 
Consequently, we will omit a proof. 
\begin{figure}[h]
\centering
  \includegraphics[width=.8\textwidth]{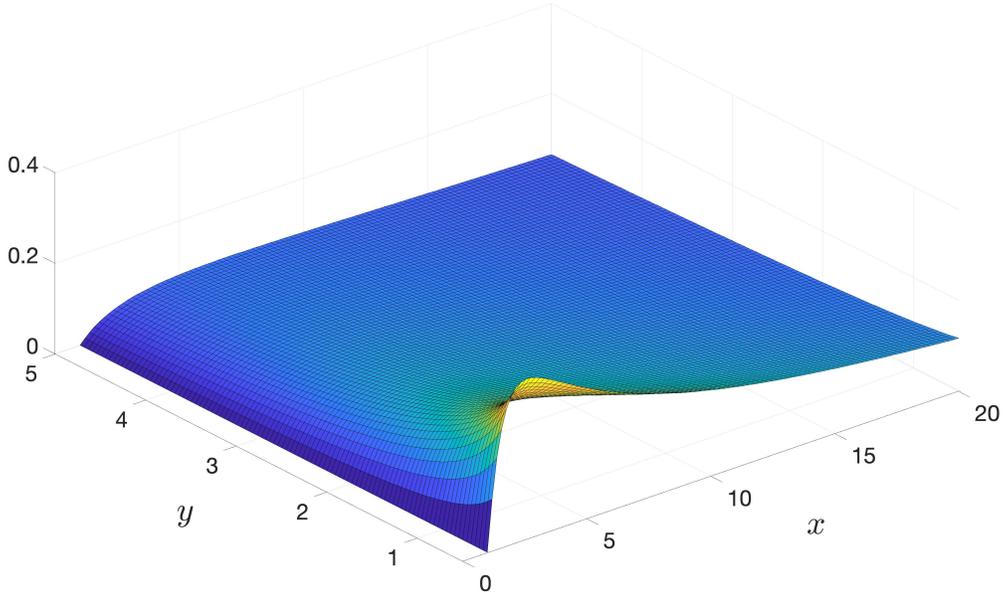}
 \caption{Numerical approximation of the function $v(x,y)$ for $(x,y)\in [1,20]\times[1/2,5]$. Here $\beta=\gamma=3$.}\label{Fig7}
\end{figure}

\section{Asymptotics}\label{AsympSect}
In this final section, we will derive a few estimates on $u(x,y)$ and $v(x,y)$ that we will need to prove Theorem \ref{AsympThm}. First, we record an upper and lower bound on $u$. 
\begin{lem}
If $x\ge 0$ and $y\ge \mu$, then
\be\label{LowerBoundExitTime}
u(x,y)\ge \frac{1}{\gamma}\ln\left(\frac{x+y}{\gamma/\beta+\mu}\right).
\ee
If $x\in [0, \gamma/\beta)$ and $y\ge \mu$, then 
\be\label{UpperBoundExitTime}
u(x,y)\le \frac{\ln(y/\mu)}{\gamma-\beta x}.
\ee
\end{lem}
\begin{proof}
Set 
$$
w(x,y)=\frac{1}{\gamma}\ln\left(\frac{x+y}{\gamma/\beta+\mu}\right).
$$
Observe that for each $x\ge 0$ and $y\ge \mu$, 
\be
\beta xy \partial_xw+(\gamma-\beta x)y\partial_yw=\frac{1}{\gamma}\frac{\beta xy+(\gamma-\beta x)y}{x+y}=
\frac{y}{x+y}\le 1.
\ee
Now fix  $x\ge 0$ and $y\ge \mu$ and suppose $S,I$ is the solution of \eqref{SIR} with $S(0)=x$ and $I(0)=y$. 
By our computation above, 
$$
\frac{d}{dt}w(S(t),I(t))=-\Big[\beta S(t)I(t)\partial_xw(S(t),I(t))+(\gamma-\beta S(t))I(t)\partial_yw(S(t),I(t))\Big]\ge -1
$$
for $0\le t\le u(x,y)$. And integrating this inequality from $t=0$ to $t=u(x,y)$ gives 
\be\label{WandUIneq}
w(S(u(x,y)),\mu)-w(x,y)\ge -u(x,y).
\ee
Since $S(u(x,y))\le \gamma/\beta$,
$$
w(S(u(x,y)),\mu)=\frac{1}{\gamma}\ln\left(\frac{S(u(x,y))+\mu}{\gamma/\beta+\mu}\right)\le 0.
$$
Combined with \eqref{WandUIneq} this implies $u(x,y)\ge w(x,y)$.  We conclude \eqref{LowerBoundExitTime}.

\par Now suppose $\beta x<\gamma$ and $y>\mu$. By \eqref{ExpFormulae},  
$$
\mu=I(u(x,y))=ye^{\displaystyle\int^{u(x,y)}_0(\beta S(\tau)-\gamma)d\tau}\le ye^{(\beta x-\gamma)u(x,y)}.
$$
Taking the  natural logarithm and rearranging gives \eqref{UpperBoundExitTime}.
\end{proof}

\par Likewise, we can identify convenient upper and lower bounds for $v(x,y)$.  To this end, we will exploit the fact that for each 
$x> \gamma/\beta$ and $y>0$ 
\be\label{geefunction}
g(z)=(\gamma/\beta)\ln z-z+\psi(x,y)
\ee
is  concave on the interval $\gamma/\beta\le z\le x.$ This implies
\be\label{wlowerBound}
g(z)\ge \frac{g(\gamma/\beta)-y}{\gamma/\beta-x}(z-x)+g(x)=\left(\frac{\gamma}{\beta}\frac{\ln x-\ln(\gamma/\beta)}{x-\gamma/\beta}-1\right)(z-x)+y
\ee
and 
\be\label{wupperBound}
g(z)\le g'(x)(z-x)+g(x)=\left(\frac{\gamma}{\beta x}-1\right)(z-x)+y
\ee
for $\gamma/\beta\le z\le x.$ 


\begin{lem}
Suppose $x> \gamma/\beta$ and $y>0$. Then 
\be\label{veeUpperBound}
v(x,y)\le \frac{\ln x-\ln(\gamma/\beta)-\ln y+\ln\left(x-\gamma/\beta+y-\frac{\gamma}{\beta}(\ln x-\ln(\gamma/\beta))\right)}{\beta\left(\displaystyle y+x\left(1-\frac{\gamma}{\beta}\frac{\ln x-\ln(\gamma/\beta)}{x-\gamma/\beta}\right)\right)}
\ee
and 
\be\label{veeLowerBound}
v(x,y)\ge \frac{\ln x-\ln(\gamma/\beta)-\ln y+\ln\left(x-\gamma/\beta+y+\frac{\gamma}{\beta}(\frac{\gamma}{\beta x}-1)\right)}{\beta(\displaystyle x-\gamma/\beta+y)}.
\ee
\end{lem}
\begin{proof}
We will appeal to part $(iii)$ of Theorem \ref{vthm} which asserts 
\be
v(x,y)=\int^x_{\gamma/\beta}\frac{dz}{\beta z g(z)}.
\ee
Here $g(z)$ is defined in \eqref{geefunction}.  We will also employ the identity 
\be\label{partialFractions}
\frac{d}{dz}\frac{\ln z-\ln(cz+d)}{d}=\frac{1}{z(cz+d)}.
\ee

\par By \eqref{wlowerBound} and \eqref{partialFractions}, 
\begin{align}
v(x,y)&\le\int^x_{\gamma/\beta}\frac{dz}{\displaystyle\beta z \left(\left(\frac{\gamma}{\beta}\frac{\ln x-\ln(\gamma/\beta)}{x-\gamma/\beta}-1\right)(z-x)+y\right)}\\
&=\left.\frac{\ln z-\ln\left(\left(\displaystyle\frac{\gamma}{\beta}\frac{\displaystyle\ln x-\ln(\gamma/\beta)}{\displaystyle x-\gamma/\beta}-1\right)(z-x)+y\right) }{\beta\displaystyle\left(y+x\left(1-\frac{\gamma}{\beta}\frac{\ln x-\ln(\gamma/\beta)}{x-\gamma/\beta}\right) \right)}\right|^{z=x}_{z=\gamma/\beta}\\
&=\frac{\ln x-\ln(\gamma/\beta)-\ln y+\ln\left(x-\gamma/\beta+y-\frac{\gamma}{\beta}(\ln x-\ln(\gamma/\beta))\right)}{\beta\left(\displaystyle y+x\left(1-\frac{\gamma}{\beta}\frac{\ln x-\ln(\gamma/\beta)}{x-\gamma/\beta}\right)\right)}.
\end{align}
Similarly, \eqref{wupperBound} and \eqref{partialFractions} give 
\begin{align}
v(x,y)&\ge\int^x_{\gamma/\beta}\frac{dz}{\displaystyle\beta z \left(\left(\frac{\gamma}{\beta x}-1\right)(z-x)+y\right)}\\
&=\left.\frac{\ln z-\ln\left(\left(\displaystyle\frac{\gamma}{\beta x}-1\right)(z-x)+y\right) }{\beta(\displaystyle x-\gamma/\beta+y)}\right|^{z=x}_{z=\gamma/\beta}\\
&=\frac{\ln x-\ln(\gamma/\beta)-\ln y+\ln\left(x-\gamma/\beta+y+\frac{\gamma}{\beta}(\frac{\gamma}{\beta x}-1)\right)}{\beta(\displaystyle x-\gamma/\beta+y)}.
\end{align}
\end{proof}
\begin{cor}
For each $\delta>0$, 
\be\label{vBoundedLim}
\lim_{\substack{x+y\rightarrow\infty \\ x\ge0, y\ge \delta}}v(x,y)=0.
\ee
\end{cor}
\begin{proof}
Choose sequences $x_k\ge 0$ and $y_k\ge \delta$ with $x_k+y_k\rightarrow\infty$ such that 
$$
\limsup_{\substack{x+y\rightarrow\infty \\ x\ge \gamma/\beta, y\ge \delta}}v(x,y)=\lim_{k\rightarrow\infty}v(x_k,y_k).
$$
If $x_k\le \gamma/\beta$ for infinitely many $k\in \N$, then $v(x_k,y_k)=0$ for infinitely many $k$ and 
\be\label{vkaygoestozero}
\lim_{k\rightarrow\infty}v(x_k,y_k)=0.
\ee
Otherwise, we may as well suppose that $x_k> \gamma/\beta$ for all $k\in \N$. In this case, \eqref{veeUpperBound} implies
  \be\label{veeUpperBound2}
v(x_k,y_k)\le \frac{\ln x_k-\ln(\gamma/\beta)-\ln y_k+\ln\left(x_k-\gamma/\beta+y_k-\frac{\gamma}{\beta}(\ln x_k-\ln(\gamma/\beta))\right)}{\beta\left(\displaystyle y_k+x_k\left(1-\frac{\gamma}{\beta}\frac{\ln x_k-\ln(\gamma/\beta)}{x_k-\gamma/\beta}\right)\right)}
\ee
for all $k\in \N$.

\par If $x_k\rightarrow \infty$, then 
$$
\frac{\gamma}{\beta}\frac{\ln x_k-\ln(\gamma/\beta)}{x_k-\gamma/\beta}\le \frac{1}{2}
$$
for sufficiently large $k$. It follows that 
\begin{align}
v(x_k,y_k)&\le \frac{\ln x_k-\ln(\gamma/\beta)-\ln y_k+\ln\left(x_k-\gamma/\beta+y_k-\frac{\gamma}{\beta}(\ln x_k-\ln(\gamma/\beta))\right)}{\beta\left(y_k+\frac{1}{2}x_k\right)}\\
&\le \frac{\ln x_k-\ln(\gamma/\beta)+\ln\left(\frac{x_k-\gamma/\beta}{y_k}+1\right)}{\beta\left(y_k+\frac{1}{2}x_k\right)}\\
&\le \frac{\ln x_k-\ln(\gamma/\beta)+\ln\left(\frac{x_k-\gamma/\beta}{\delta}+1\right)}{\beta\left(\delta+\frac{1}{2}x_k\right)}
\end{align}
for all large enough $k$. Therefore, \eqref{vkaygoestozero} holds.  

\par Alternatively, we can pass to a subsequence if necessary and suppose $x_k\le c$ for all $k\in \N$ and $y_k\rightarrow\infty$. Lemma \ref{CrudeUpperV} then gives
\be
v(x_k,y_k)\le \frac{\ln c-\ln(\gamma/\beta)}{\beta y_k}\rightarrow 0.
\ee
As a result, \eqref{vkaygoestozero} holds in all cases.  It follows that 
$$
\limsup_{\substack{x+y\rightarrow\infty \\ x\ge 0, y\ge \delta}}v(x,y)=0,
$$
which in turn implies \eqref{vBoundedLim}.
\end{proof}
We are now ready to prove Theorem \ref{AsympThm} which asserts 
\be\label{uAsymptotic2}
\lim_{\substack{x+y\rightarrow\infty \\ x\ge 0,\;y\ge \mu}}\frac{\displaystyle u(x,y)}{\displaystyle\frac{1}{\gamma}\ln\left(\frac{x+y}{\mu}\right)}=1
\ee
and 
\be\label{vAsymptotic2}
\lim_{\substack{x+y\rightarrow\infty \\ x> \gamma/\beta,\;y> 0}}\frac{\displaystyle \beta(x-\gamma/\beta+y) }{\displaystyle\ln\left[\left(\frac{x}{\gamma/\beta}\right)\left(\frac{x-\gamma/\beta}{y}+1\right)\right]}\cdot v(x,y)=1. 
\ee
\begin{proof}[Proof of \eqref{uAsymptotic2}]
In view of \eqref{LowerBoundExitTime}, 
$$
\liminf_{\substack{x+y\rightarrow\infty \\ x\ge 0\; y\ge \mu}}\frac{\displaystyle u(x,y)}{\displaystyle\frac{1}{\gamma}\ln\left(\frac{x+y}{\mu}\right)}\ge1. 
$$
It follows that $u(x,y)\rightarrow \infty$ as $x+y\rightarrow\infty$ with $x\ge 0$ and $y\ge \mu$. In view of Corollary \ref{vBoundedLim}, we may select $N\in \N$ so large that 
$$
v(x,y)< u(x,y)
$$
for all $x+y\ge N$ with $x\ge 0$ and $y\ge \mu$.  

\par Suppose $x+y\ge N$ with $x\ge 0$ and $y\ge \mu$ and choose a time $t$ such that
$$
v(x,y)<t<u(x,y).
$$
Note that as $t>v(x,y)$,
$$
S(t)<\frac{\gamma}{\beta}.
$$
Here $S,I$ is the solution of the SIR ODE \eqref{SIR} with $S(0)=x$ and $I(0)=y$.  By \eqref{UpperBoundExitTime}, we also have
\begin{align}\label{UpperBoundExitTime22}
u(x,y)&=t+u(S(t),I(t))\nonumber\\
&\le t+\frac{1}{\gamma-\beta S(t)}\ln\left(\frac{I(t)}{\mu}\right)\nonumber\\
&\le t+\frac{1}{\gamma-\beta S(t)}\ln\left(\frac{x+y}{\mu}\right).
\end{align}

\par In addition, we can use \eqref{ExpFormulae} to find
\begin{align}
S(t)&=S(v(x,y))e^{\displaystyle-\beta \int^t_{v(x,y)}I(\tau)d\tau}\\
&=\frac{\gamma}{\beta}e^{\displaystyle-\beta \int^t_{v(x,y)}I(\tau)d\tau}\\
&\le \frac{\gamma}{\beta}e^{-\beta\mu(t-v(x,y))}.
\end{align}
Here we used that $I(\tau)\ge \mu$ when $\tau\le u(x,y)$. Therefore, 
$$
\gamma-\beta S(t)\ge \gamma \left(1-e^{-\beta\mu(t-v(x,y))}\right)
$$
Combining this inequality with \eqref{UpperBoundExitTime22} gives 
\be
u(x,y)\le t + \frac{1}{1-e^{-\beta\mu(t-v(x,y))}}\frac{1}{\gamma}\ln\left(\frac{x+y}{\mu}\right).
\ee
As a result, 
$$
\limsup_{\substack{x+y\rightarrow\infty \\ x\ge 0\;y\ge \mu}}\frac{\displaystyle u(x,y)}{\displaystyle\frac{1}{\gamma}\ln\left(\frac{x+y}{\mu}\right)}\le \frac{1}{1-e^{-\beta\mu t}}. 
$$
We conclude \eqref{uAsymptotic2} upon sending $t\rightarrow\infty$.
\end{proof}
In our closing argument below, we will employ the elementary inequalities 
\be\label{2logInequalities}
\frac{1}{x}\le\frac{\ln x-\ln(\gamma/\beta)}{x-\gamma/\beta} \le \frac{\beta}{\gamma},
\ee
which hold for $x>\gamma/\beta$. They follow as the natural logarithm is concave. 
\begin{proof}[Proof of \eqref{vAsymptotic2}]
By the upper bound \eqref{veeUpperBound},
\begin{align}\label{upperBdVLimitStep}
&\frac{\displaystyle \beta(x-\gamma/\beta+y) }{\displaystyle\ln\left[\left(\frac{x}{\gamma/\beta}\right)\left(\frac{x-\gamma/\beta}{y}+1\right)\right]}\cdot v(x,y)\nonumber\\
&\quad =\frac{\displaystyle \beta(x-\gamma/\beta+y) }{\displaystyle \ln x-\ln(\gamma/\beta)-\ln y+\ln\left(x-\gamma/\beta+y\right)}\cdot v(x,y)\nonumber\\
&\quad \le \frac{\displaystyle \beta(x+y) }{\ln x-\ln(\gamma/\beta)-\ln y+\ln\left(x-\gamma/\beta+y-\frac{\gamma}{\beta}(\ln x-\ln(\gamma/\beta))\right)}\cdot v(x,y)\nonumber\\
&\quad \le \frac{\displaystyle x+y}{\displaystyle y+x\left(1-\frac{\gamma}{\beta}\frac{\ln x-\ln(\gamma/\beta)}{x-\gamma/\beta}\right)}\nonumber\\
&\quad = \frac{\displaystyle 1}{\displaystyle 1-\frac{x}{x+y}\frac{\gamma}{\beta}\left(\frac{\ln x-\ln(\gamma/\beta)}{x-\gamma/\beta}\right)}.
\end{align}
We may select sequences $x_k>\gamma/\beta$ and $y_k> 0$ with $x_k+y_k\rightarrow\infty$ and
\be\label{upperBdVLimitStep2}
\limsup_{\substack{x+y\rightarrow\infty \\ x> \gamma/\beta,\;y> 0}}
\frac{x}{x+y}\frac{\gamma}{\beta}\left(\frac{\ln x-\ln(\gamma/\beta)}{x-\gamma/\beta}\right)=
\lim_{k\rightarrow\infty}
\frac{x_k}{x_k+y_k}\frac{\gamma}{\beta}\left(\frac{\ln x_k-\ln(\gamma/\beta)}{x_k-\gamma/\beta}\right).
\ee

\par If $x_k\rightarrow\infty$, then 
$$
0\le \frac{x_k}{x_k+y_k}\frac{\gamma}{\beta}\left(\frac{\ln x_k-\ln(\gamma/\beta)}{x_k-\gamma/\beta}\right)
\le\frac{\gamma}{\beta} \frac{\ln x_k}{x_k-\gamma/\beta}\rightarrow0.
$$
Alternatively, $x_k$ has a bounded subsequence. Passing to a subsequence if necessary, we may assume that $x_k\le c$ for some constant $c$.  In which case, $y_k\rightarrow\infty$. Employing 
\eqref{2logInequalities}, we find
$$
0\le \frac{x_k}{x_k+y_k}\frac{\gamma}{\beta}\left(\frac{\ln x_k-\ln(\gamma/\beta)}{x_k-\gamma/\beta}\right)
\le\frac{x_k}{x_k+y_k}\le \frac{c}{y_k} \rightarrow0.
$$
It follows that the limit in \eqref{upperBdVLimitStep2} is $0$.  And in view of \eqref{upperBdVLimitStep}, 
$$
\limsup_{\substack{x+y\rightarrow\infty \\ x> \gamma/\beta,\;y> 0}}\frac{\displaystyle \beta(x-\gamma/\beta+y) }{\displaystyle\ln\left[\left(\frac{x}{\gamma/\beta}\right)\left(\frac{x-\gamma/\beta}{y}+1\right)\right]}\cdot v(x,y)\le 1.
$$

\par By the lower bound \eqref{veeLowerBound},
\begin{align}\label{upperLdVLimitStep2}
\frac{\displaystyle \beta(x-\gamma/\beta+y) v(x,y) }{\displaystyle\ln\left[\left(\frac{x}{\gamma/\beta}\right)\left(\frac{x-\gamma/\beta}{y}+1\right)\right]}\nonumber
& =\frac{\displaystyle \beta(x-\gamma/\beta+y)v(x,y) }{\displaystyle \ln x-\ln(\gamma/\beta)-\ln y+\ln\left(x-\gamma/\beta+y\right)}\nonumber\\
&\quad\ge \frac{\ln x-\ln(\gamma/\beta)-\ln y+\ln\left(x-\gamma/\beta+y+\frac{\gamma}{\beta}(\frac{\gamma}{\beta x}-1)\right)}{\displaystyle \ln x-\ln(\gamma/\beta)-\ln y+\ln\left(x-\gamma/\beta+y\right)}
\nonumber\\
&\quad= 1+\frac{\displaystyle\ln\left(\frac{x-\gamma/\beta+y+\frac{\gamma}{\beta}(\frac{\gamma}{\beta x}-1)}{x-\gamma/\beta+y}\right)}{\displaystyle \ln x-\ln(\gamma/\beta)-\ln y+\ln\left(x-\gamma/\beta+y\right)}\nonumber\\
&\quad= 1+\frac{\displaystyle\ln\left(1+\frac{\frac{\gamma}{\beta}(\frac{\gamma}{\beta x}-1)}{x-\gamma/\beta+y}\right)}{\displaystyle \ln\left(\frac{x}{\gamma/\beta}\right)+\ln\left(\frac{x-\gamma/\beta}{y}+1\right)}.
\end{align}
Let us choose sequences $x_k>\gamma/\beta$ and $y_k> 0$ such that $x_k+y_k\rightarrow\infty$ and
\be\label{upperLdVLimitStep3}
\liminf_{\substack{x+y\rightarrow\infty \\ x> \gamma/\beta,\;y> 0}}
\frac{\displaystyle\ln\left(1+\frac{\frac{\gamma}{\beta}(\frac{\gamma}{\beta x}-1)}{x-\gamma/\beta+y}\right)}{\displaystyle \ln\left(\frac{x}{\gamma/\beta}\right)+\ln\left(\frac{x-\gamma/\beta}{y}+1\right)}=
\lim_{k\rightarrow\infty}
\frac{\displaystyle\ln\left(1+\frac{\frac{\gamma}{\beta}(\frac{\gamma}{\beta x_k}-1)}{x_k-\gamma/\beta+y_k}\right)}{\displaystyle \ln\left(\frac{x_k}{\gamma/\beta}\right)+\ln\left(\frac{x_k-\gamma/\beta}{y_k}+1\right)}.
\ee

\par We recall that $\ln(1+z)/z\rightarrow 1$ as $z\rightarrow 0$, which implies 
$$
\ln(1+z)\ge \frac{3}{2}z
$$
for all nonpositive $z$ sufficiently close to $0$. Since
$$
0\ge\frac{\displaystyle\frac{\gamma}{\beta}\left(\frac{\gamma}{\beta x_k}-1\right)}{x_k-\gamma/\beta+y_k}\rightarrow 0,
$$
we then have 
$$
\ln\left(1+\frac{\frac{\gamma}{\beta}(\frac{\gamma}{\beta x_k}-1)}{x_k-\gamma/\beta+y_k}\right)
\ge \frac{3}{2}\frac{\displaystyle\frac{\gamma}{\beta}\left(\frac{\gamma}{\beta x_k}-1\right)}{x_k-\gamma/\beta+y_k}
$$
for all sufficiently large $k\in \N$. 

\par Furthermore,
\begin{align}
0&\ge \frac{\displaystyle \frac{2}{3} \ln\left(1+\frac{\frac{\gamma}{\beta}(\frac{\gamma}{\beta x_k}-1)}{x_k-\gamma/\beta+y_k}\right)}{\displaystyle \ln\left(\frac{x_k}{\gamma/\beta}\right)+\ln\left(\frac{x_k-\gamma/\beta}{y_k}+1\right)}\\
&= \frac{ \frac{\displaystyle\frac{\gamma}{\beta}\left(\frac{\gamma}{\beta x_k}-1\right)}{\displaystyle x_k-\gamma/\beta+y_k}}
{\displaystyle \ln\left(\frac{x_k}{\gamma/\beta}\right)+\ln\left(\frac{x_k-\gamma/\beta}{y_k}+1\right)}\\
&\ge\frac{1}{\displaystyle\ln\left(\frac{x_k}{\gamma/\beta}\right)}\frac{\displaystyle\frac{\gamma}{\beta}\left(\frac{\gamma}{\beta x_k}-1\right)}{x_k-\gamma/\beta+y_k}\\
&=-\frac{1}{x_k}\frac{x_k-\gamma/\beta}{\ln x_k-\ln(\gamma/\beta)}\frac{\gamma/\beta }{x_k-\gamma/\beta+y_k}\\
&\ge -\frac{\gamma/\beta }{x_k-\gamma/\beta+y_k}.
\end{align}
In the last inequality, we used \eqref{2logInequalities}. We conclude the limit in \eqref{upperLdVLimitStep3} is $0$.  In view of inequality \eqref{upperLdVLimitStep2},
$$
\liminf_{\substack{x+y\rightarrow\infty \\ x> \gamma/\beta,\;y> 0}}\frac{\displaystyle \beta(x-\gamma/\beta+y)  }{\displaystyle\ln\left[\left(\frac{x}{\gamma/\beta}\right)\left(\frac{x-\gamma/\beta}{y}+1\right)\right]}\cdot v(x,y)
\ge 1.$$
\end{proof}
\bibliography{SIRbib}{}

\bibliographystyle{plain}

\typeout{get arXiv to do 4 passes: Label(s) may have changed. Rerun}

\end{document}